\documentclass[12pt]{amsart}

\usepackage[english]{babel} 
\usepackage[utf8]{inputenc}

\usepackage{enumitem}
\usepackage{amsfonts}
\usepackage{amsmath}
\usepackage{amssymb}
\usepackage{amsthm}
\usepackage{mathtools}
\usepackage{stmaryrd}
\usepackage{hyperref}
\usepackage{mathrsfs}
\hypersetup{
    colorlinks=false,
    linkcolor=blue
}
\usepackage[capitalize]{cleveref}
\usepackage{tikz}
\input tikz.tex
\usetikzlibrary{cd}

\usepackage{graphicx}
\usepackage[letterpaper, margin=1in, total={6in, 8in}]{geometry}

\newtheorem{thm}{Theorem}[section]
\newtheorem{lem}[thm]{Lemma}
\newtheorem{cor}[thm]{Corollary}
\newtheorem{prop}[thm]{Proposition}
\newtheorem{rmk}[thm]{Remark}

\theoremstyle{definition}

\theoremstyle{remark}
\newtheorem{rem}[thm]{Remark}

\newcommand{\Z}{\mathbb Z}

\newcommand{\M}{\mathcal{M}}

\renewcommand{\O}{\mathcal{O}}

\newcommand{\Proj}{\operatorname{Proj}}

\newcommand{\FF}{\mathbb{F}}
\newcommand{\PP}{\mathbf{P}}

\newcommand{\PGL}{\operatorname{PGL}}

\renewcommand{\phi}{\varphi}

\usepackage[OT2,T1]{fontenc}
\DeclareSymbolFont{cyrletters}{OT2}{wncyr}{m}{n}
\DeclareMathSymbol{\Sha}{\mathalpha}{cyrletters}{"58}

\nocite{*}

\title{A Census of Genus 6 Curves over $\FF_2$}
\author{Yongyuan Huang, Kiran S. Kedlaya, and Jun Bo Lau}
\date{\today}
\thanks{Kedlaya was supported by NSF (grant DMS-2053473), UC San Diego (Warschawski professorship), and the Simons Foundation (Simons Fellowship in Mathematics, 2023--24). Lau is supported by the Simons Foundation grant \#550023 for the Collaboration on Arithmetic Geometry, Number Theory, and
Computation. Computational resources were provided by the Simons Collaboration in Arithmetic Geometry, Number Theory, and Computation and by Boston University's Department of Mathematics and Statistics.}

\begin{document}

\begin{abstract}
We compile a complete list of isomorphism class representatives of curves of genus 6 over $\FF_2$. We use explicit descriptions of canonical curves in each stratum of the Brill--Noether stratification of the moduli space $\mathcal{M}_6$, due to Mukai in the generic case. Our computed value of $\#\mathcal{M}_6(\mathbb{F}_2)$ agrees with the Lefschetz trace formula as recently computed by Bergstrom--Canning--Petersen--Schmitt.
\end{abstract}

\maketitle

\section{Introduction}

For $g>1$, let $\M_g$ denote the moduli stack of curves of genus $g$. 
(All ``curves'' herein are smooth, projective, and geometrically irreducible unless otherwise specified.)
For each prime power $q$, the set $\M_g(\FF_q)$ of $\FF_q$-valued points of $\M_g$ is finite; it is naturally identified with the set of isomorphism classes of curves of genus $g$ over $\FF_q$. 
We equip $\M_g(\FF_q)$ with the measure 
which gives the isomorphism class of a curve $C$ the weight $\frac{1}{\#\mathrm{Aut}(C)}$, as in the Lefschetz trace formula for Deligne--Mumford stacks \cite{Behrend}. (Here we count automorphisms of $C$ over $\FF_q$ itself, not its base extension to an algebraic closure.)

Since $\M_g$ has relative dimension $3g-3$ over $\Z$, it is feasible to compute the set $\M_g(\FF_q)$ for small values of $g$ and $q$, especially for $q=2$ where this has been done previously for $g \leq 5$ \cite{Xarles, Dragutinovic}. In this paper, we extend the computation to the case $g=6$.
\begin{thm} \label{T:main}
We obtain an explicit list of isomorphism class representatives for $\M_6(\FF_2)$: it consists of $72227$ elements, and
\begin{equation} \label{eq:full count over F2}
\#\M_6(\FF_2) = 68615.
\end{equation}
\end{thm}
A list of isomorphism class representatives, as well as the SageMath \cite{SageMath} and Magma \cite{Magma} code used to generate it, can be found at
\begin{center}
\url{https://github.com/junbolau/genus-6}.
\end{center}
The list is also available via the table of isogeny classes of abelian varieties over finite fields in LMFDB \cite{LMFDB}. We observe that 38327 of the 164937 isogeny classes of abelian sixfolds over $\FF_2$ contain at least one Jacobian, representing all 20 of the possible Newton polygons, and that the maximum number of Jacobians in a single isogeny class is 20.

Our approach to Theorem~\ref{T:main} follows the partial census carried out in \cite{Kedlaya}: for each stratum in the Brill--Noether stratification of $\M_6$, we use the descriptions of general canonical curves in each stratum (due to Mukai \cite{Mukai} for the generic stratum) to construct a covering set for the isomorphism classes of curves over $\FF_2$ in that stratum. We then make extensive use of Magma's implementation of function fields to identify isomorphic curves and compute automorphism groups; the only groups that occur are
\[
C_1, C_2, C_3, C_4, C_2 \times C_2, C_5, C_6, S_3, C_{10}, D_5, D_{10}, A_5.
\]

We have two main applications in mind for Theorem~\ref{T:main}. One is to identify curves with a given zeta function; for example, the following statements can now be verified by database queries in LMFDB.
\begin{itemize}
    \item The maximum number of $\FF_2$-points on a curve of genus 6 is 10, achieved by exactly two curves \cite{Rigato}.
    \item There are 70 supersingular curves of genus 6 over $\FF_2$, with 28 distinct zeta functions.
    \item There is no curve of genus 6 over $\FF_2$ having any of the three zeta functions listed in the proof of \cite[Theorem~5.1]{FaberGranthamHowe}; that result states that the maximum gonality of a curve of genus 6 over $\FF_2$ is 6.
    \item There is no curve $C$ of genus 6 over $\FF_2$ with $\#C(\FF_{2^4}) = 0$. This recovers the previous assertion as well as a nonexistence statement made in \cite[\S 6]{Kedlaya-rel1}.
    \item There is a unique curve $C$ of genus 6 over $\FF_2$ with $(\#C(\FF_{2^i}))_{i=1}^6 = (0,0,0,20,15,90)$ \cite[Lemma 10.2]{Kedlaya-rel2}.
    \item As reported in \cite[Table~1]{Kedlaya}, there are 52 curves with zeta functions matching one of the options in \cite[Table~2]{Kedlaya}. The latter contains (as shown in \cite{Kedlaya-rel2}) every curve of genus 6 admitting an \'etale double cover with trivial relative class group.
\end{itemize}

The other intended application is to the computation of the rational cohomology of $\M_g$. 
There has been much progress in this direction recently;
for example, it is known that $\#\M_g(\FF_q)$ is a polynomial in $q$ for each $g \leq 6$ \cite[Corollary~1.6]{CanningLarson-Mg}.
More precisely, this follows from the Lefschetz trace formula and the fact that in these cases, the rational cohomology of $\M_g$ can be computed using the tautological Chow ring. The latter can be computed using the Sage package described in \cite{admcycles}; by so doing, one can recover the explicit polynomials for $g=4$ (see \cite[\S 4]{BergstromTommasi} or \cite[Theorem 1.5]{BergstromFaberPayne}
 for $g=4$ and \cite{BergstromEtAl} for $g=5,6$).
The resulting formula for $g=6, q=2$ agrees with \eqref{eq:full count over F2}; while Theorem~\ref{T:main} is in principle logically independent of this agreement, admitting it allows for an alternate justification of the correctness of our result (see \S\ref{sec:completeness}).

On the other hand, a tabulation of curves of genus $g$ also yields, for every positive integer $n$, a point count for the stack $\M_{g,n}$ of $n$-pointed genus $g$ curves (where the points are distinct and distinguishable) or more generally any quotient of $\M_{g,n}$ by a subgroup of $S_n$.
For example, Theorem~\ref{T:main} yields the following.
\begin{cor} \label{cor:count points on marked}
We have
\begin{align}
    \#\M_{6,1}(\FF_2) &= 223317, \\
    \#(\M_{6,2}/S_2)(\FF_2) &= 471210, \\
    \#\M_{6,2}(\FF_2) &= 650838.
\end{align}
\end{cor}
In these cases, \cite[Theorem~1.4]{CanningLarson-Mg} implies that the point count over $\FF_q$ is a polynomial in $q$, but as of now the computation of these polynomials remains infeasible using \cite{admcycles}. Our computation provides one linear constraint on the coefficients of the polynomial, and thus reduces by one the number of rational cohomology groups that need to be computed in order to determine the polynomial. (One could adapt our methods to perform a census over $\FF_3$ and thus obtain a second linear constraint; we do not plan to do this.)

We observe that \cite{Kedlaya} also includes a partial census of genus 7 curves over $\FF_2$, which it should be possible to similarly upgrade to a full census. A polynomial formula for $\#\overline{\M}_7(\FF_q)$, where $\overline{\M}_7$ denotes the moduli stack of stable curves of genus 7, will appear in \cite{BergstromEtAl}; combining this formula with Corollary~\ref{cor:count points on marked} and known polynomial formulas for $\#M_{g,n}(\FF_q)$ for $g \leq 5$ will yield the value of $\#\M_7(\FF_2)$. 

It is unclear whether one can push this further, say to genus 8 or even 9.
On one hand, the expected number of curves (approximately $2^{3g-3}$ in genus $g$)
is manageable, and we again have explicit descriptions of canonical curves in these genera \cite{IdeMukai, Mukai2, Mukai3}. On the other hand, these descriptions are currently only available over an algebraically closed based field; moreover, while we expect a polynomial formula for $\#\overline{\M}_7(\FF_q)$ to be obtained in \cite{BergstromEtAl}, it is unclear whether
$\#\overline{\M}_g(\FF_q)$ admits a polynomial formula for $g=8$ or $g=9$, and even more unclear whether these quantities can be computed using current technology.

\section{The Brill--Noether stratification of \texorpdfstring{$\M_6$}{M6}}

We first recall some relevant terminology and facts about $\M_6$.
Throughout this discussion, let $C$ be a curve of genus $g$ over a finite field $k$ and let $\overline{k}$ be an algebraic closure of $k$.
Let $K$ be the canonical divisor on $C$, and $|K|$ be the canonical linear system. 

A \emph{$g^r_d$} on $C$ is a linear system of dimension $r$ of degree $d$, which if basepoint-free defines a degree $d$ morphism $C\to\PP^r$. We call $C$ \textit{hyperelliptic} if there is a finite morphism $C\to\PP^1$ of degree 2, or equivalently if $C$ admits a $g^1_2$ (which is automatically basepoint-free if $g\geq 1$). We call the morphism $\iota\colon C\to\PP^{g-1}_k$ defined by $|K|$ the \textit{canonical morphism}. For $g\geq 2$, $|K|$ is very ample if and only if $C$ is not hyperelliptic. Thus if $C$ is nonhyperelliptic, the canonical morphism $\iota$ is an embedding, and if $C$ is hyperelliptic, $\iota$ factors as a degree $2$ morphism $C\to\PP^1$ followed by the Veronese embedding $\PP^1_k\to\PP^{g-1}_k$. In particular, every genus two curve is hyperelliptic.    

For $g\geq 4$, we say $C$ is \textit{trigonal} if $C$ admits a $g^1_3$ but not a $g^1_2$; let $\mathcal{T}_g$ be the stack of smooth trigonal curves. The moduli space $\mathcal{T}_g$ admits a stratification by locally closed substacks $\mathcal{T}_{g,n}$ where $n$ runs over integers with $0 \leq n \leq \frac{g+2}{3}$ and $n \equiv g \pmod{2}$. The integer $n$ denotes the 
\textit{Maroni invariant} of a trigonal curve $C$, defined as the unique nonnegative integer $n$ such that the trigonal cover $C\to\PP^1$ factors through a closed embedding 
\[
C\to \mathbf{F}_n :=\PP_{\PP^1}(\O_{\PP^1}\oplus\O(n)_{\PP^1})
\]
in such a way that the structure map $\mathbf{F}_n\to\PP^1$ restricts to the trigonal cover $C\to\PP^1$.
Note that $\mathbf{F}_n$ is ruled by the fibers of the structure map;
it is in fact the $n$th Hirzebruch surface (which for $n=0$ degenerates to $\PP^1_{k}\times_{k}\PP^1_{k}$), and can also be represented as an $(n,1)$-hypersurface in $\PP^1_k\times_{k}\PP^2_k$.

Last but not least, we say $C$ is \textit{bielliptic} if it admits a degree $2$ map to a genus $1$ curve over $k$. Any such map gives rise to a $g^1_4$, but not conversely.

Due to work of Petri and Mukai, we have the following classification of genus $6$ curves over finite fields. 
\begin{thm}\label{genus6}
    Let $C$ be a curve of genus $6$ over a finite field $k$. Then one (and only one) of the following holds. 
    \begin{enumerate}[label=(\arabic*)]
        \item The curve $C$ is hyperelliptic. 
        \item The curve $C$ is bielliptic. 
        \item The curve $C$ occurs as a smooth quintic in $\PP^2_k$.
        \item The curve $C$ is trigonal of Maroni invariant $0$. In this case, $C$ occurs as a curve of bidegree $(3,4)$ in $\PP^1_k\times\PP^1_k$.
        \item The curve $C$ is trigonal of Maroni invariant $2$. In this case, $C$ occurs as a complete intersection of type $(2,1)\cap (1,3)$ in $\PP^1_k\times_k\PP^2_k$, where the $(2,1)$-hyperplane is isomorphic to the Hirzebruch surface $\mathbf{F}_2$.  
        \item The curve $C$ occurs as a transverse intersection of four hyperplanes, a quadric hypersurface, and the $6$-dimensional Grassmannian $\operatorname{Gr}(2,5)$ in $\PP_k^9$. 
    \end{enumerate}
    \begin{proof}
        Most of the above follows from Petri's theorem. For the details in the last case, see \cite[Theorem 3.1]{Kedlaya}.
    \end{proof}
\end{thm}

\begin{rem}
Curves as in case 6 of Theorem \ref{genus6} are known as \textit{Brill--Noether-general} curves (c.f. \cite[Theorem 4.1]{Penev-Vakil}). We will henceforth refer to them as \textit{generic} curves of genus 6.
\end{rem}

\begin{rem}
    As stated in \cite[Theorem 4.1]{Penev-Vakil}, the space $\M_6$ can be stratified into locally closed substacks consisting of the loci corresponding to each of the cases in Theorem \ref{genus6}. In particular:
    \begin{enumerate}[label=(\arabic*)]
        \item The locus $\mathcal{H}_6$ of hyperelliptic curves of genus $6$ has dimension 11. 
        \item The locus $\mathcal{B}_6$ of bielliptic curves of genus $6$ has dimension 10. 
        \item The locus $\mathcal{Q}_6$ of smooth plane quintic curves of genus $6$ has dimension 12.
        \item The locus $\mathcal{T}_{6,0}$ of trigonal curves of genus $6$ with  Maroni invariant $0$ has dimension 13.
        \item The locus $\mathcal{T}_{6,2}$ of trigonal curves of genus $6$ with Maroni invariant 2 has dimension 12.
        \item The locus $\M_6^{\operatorname{BN}}$ of generic curves of genus $6$ has dimension 15 (it is open in $\M_6$).
    \end{enumerate}
\end{rem}

\section{Tabulation of data}
\label{sec:tabulation}

We begin by recording a few convenient facts that allow us to more efficiently search and filter putative genus 6 curves. 
\begin{enumerate}
    \item Using an analogue of the explicit formula from analytic number theory, Serre (c.f. \cite[Theorem 5.3.2, 7.1 Table 1]{Serre} shows that a curve of genus 6 has at most 10 $\FF_2$-points, which is a notable refinement from the Hasse-Weil bound 15. 
    \item LMFDB contains a complete list of isogeny classes of abelian varieties of dimension 6 over $\FF_2$ and their corresponding $L$-polynomials. Using the fact that a curve and its Jacobian have the same Weil polynomial, we recover a finite set containing the tuple $(\#C(\FF_{2^i}))_{i=1}^6$ for any curve $C$ of genus 6 over $\FF_2$. The relevant code written in SageMath can be found in \verb|./Census/Shared/weil_poly_utils.sage| in our code base (taken from \cite{Kedlaya}). We make use of this list when it would presumptively speed up our tabulation process. 
\end{enumerate}

In several cases, we use the \emph{orbit lookup trees} introduced by the second author (see \cite[Appendix A]{Kedlaya}) to efficiently compute orbit representatives for the action of a group $G$ on $k$-element subsets of a finite set $S$ equipped with a left $G$-action for small values of $k$. The implementation of this algorithm in SageMath can be found in \verb|./Census/Shared/orbits.sage| in our code base (again taken from \cite{Kedlaya}).

To simplify the code somewhat, initially we only construct a finite set of genus 6 curves 
over $\FF_2$ which meets every isomorphism and is ``not too redundant''. We use a separate postprocessing step to remove redundancies (see \S\ref{subsec:postprocessing}).

\subsection{Hyperelliptic curves}
Here we follow the strategy used in \cite{Xarles,Dragutinovic} where the enumerations were done in cases $g = 4,5$. This strategy is adapted to characteristic 2; for a good approach in odd characteristic, see \cite{Howe}.

Any hyperelliptic curve of genus $g$ over $\FF_2$ can be represented as $y^2 + q(x)y = p(x)$ with $p(x),q(x) \in \FF_2[x]$ and $2g + 1 \leq \max \{ 2 \deg(q(x)) , \deg(p(x))\} \leq 2g + 2$. Xarles presented a method to determine the isomorphism class of a hyperelliptic curve using the action of $\PGL_2(\FF_2)$ on $\FF_2[x]_{\leq g+1}$.

\begin{lem}{(\cite{Xarles}, Lemma 1)}
    Let $H_1,H_2$ be hyperelliptic curves represented by the equations $y^2 + q_i(x)y = p_i(x)$ for $i \in \{1,2\}$ respectively as above. Suppose that $H_1 \cong H_2$. Then there exists $A \in \PGL_2(\FF_2)$ such that $q_2(x) = \psi_{g+1}(A)(q_1(x))$, where the action of $\PGL_2(\FF_2)$ on $\FF_2[x]_{\leq n}$ is given by
    \[
    \psi_{n}(A)(q(x)) := (cx+d)^n q\left(\frac{ax+b}{cx+d}\right), \qquad A = \begin{pmatrix}
        a & b \\ c & d
    \end{pmatrix} \in \PGL_2(\FF_2).
    \]
\end{lem}

We compute orbit representatives for this action, test for pairwise isomorphism, and record the resulting curves. The implementation of this method can be found in \verb|./Census/hyperelliptic/| in our code base. 

\subsection{Bielliptic curves}

Here we follow the strategy used in \cite{Kedlaya} to enumerate bielliptic curves of genus 7;
therein bielliptic curves of genus 6 were ruled out without any enumeration, but the enumeration strategy is genus-independent.

By Riemann--Hurwitz plus the fact that double covers in characteristic 2 have only wild ramification, the map from a bielliptic curve $C$ of genus 6 to its elliptic quotient $E$ has ramification divisor of the form $2D$ where $D$ is an effective divisor of degree $g-1 = 5$ on $E$.
We may thus generate all bielliptic curves by enumerating over a set of isomorphism class representatives of elliptic curves $E$ over $\FF_2$ (there are 5 of them). For each $E$, we use Magma to enumerate over all effective divisors $D$ of degree 5. For each $D$, we enumerate over all order-2 quotients of the ray class group of $D$, form the corresponding abelian extension, then check to see if it indeed has genus 6 (and if so record the resulting curve). The implementation of this method can be found in \verb|./Census/bielliptic/| in our code base. 

\subsection{Smooth plane quintic curves}

Since the space of quintic polynomials over $\FF_2$ has dimension $\binom{7}{2} = 21$, it is not necessary to reduce this space using the action of $\mathrm{GL}(3, \FF_2)$; we simplify identify all of the nonsingular polynomials and record the resulting smooth curves. The implementation of this method can be found in \verb|./Census/plane_quintic/| in our code base.

\subsection{Trigonal curves of Maroni invariant 0}

In this case, we are looking for $(3,4)$-curves in $\PP^1 \times \PP^1$, and we follow the strategy used in \cite{Kedlaya}. We first compute orbit representatives for the action of $\PGL_2(\FF_2) \times \PGL_2(\FF_2)$ on \emph{all} subsets of $(\PP^1 \times \PP^1)(\FF_2)$. For each orbit representative, we identify the $(3,4)$-polynomials which vanish on the points in the chosen subset and do not vanish elsewhere; since we are working over $\FF_2$, this is an affine subspace of the vector space of $(3,4)$-polynomials.
We then pick out the nonsingular polynomials and record the resulting smooth curves. The implementation of this method can be found in \verb|./Census/trigonal_maroni_0/| in our code base.

\subsection{Trigonal curves of Maroni invariant 2}

In this case, we are looking for complete intersections of type $(2,1) \cap (1,3)$ in $\PP^1 \times \PP^2$, specifically, if we write $\PP^1 \times \PP^2 = \Proj \FF_2[x_0, x_1; y_0, y_1, y_2]$, then we may take the $(2,1)$-hypersurface $X_1$ to be
\begin{equation}\label{F2}
(x_0^2 + x_1^2) y_1 + x_0 x_1 y_2 = 0:
\end{equation}
over a field of characteristic 0, the equation of the Hirzebruch surface $\mathbf{F}_n$ is isomorphic to the hypersurface defined by $x_0^ny_1-x_1^ny_2$ in $\PP^1\times\PP^2$ (c.f. \cite{Huybrechts} Exercise 2.4.5), and we obtain \eqref{F2} by taking $n=2$ and making a change of variables to get an equation with smooth mod-2 reduction.

The hypersurface (\ref{F2}) is fixed by the group $G$ generated by the three involutions
\[
x_0 \leftrightarrow x_1; \qquad
y_0 \mapsto y_0+y_1; \qquad
y_0 \mapsto y_0+y_2.
\]
We now proceed as in the previous case: we compute orbit representatives for the action of $G$ on all subsets of $X_1$; for each orbit representative, we identify the $(1,3)$-polynomials which vanish on the points in the chosen subset and do not vanish elsewhere; we then pick out the nonsingular polynomials and record the resulting smooth curves. The implementation of this method can be found in \verb|./Census/trigonal_maroni_2/| in our code base.

\subsection{Generic curves}

Here we follow a modified version of the strategy used in \cite{Kedlaya}. This is the most computationally intensive case.
We first identify orbit representatives for the action of $\mathrm{PGL}_5(\FF_2)$ on 4-tuples of points in ${\PP^9}^{\vee}(\FF_2)$. Each 4-tuple defines 4 linear forms and hence 4 hyperplanes on $\PP^9$; we next compute representatives for the linear action of 
$\mathrm{PGL}_4(\FF_2)$ on such 4-tuples preserving the intersection of the 4 hyperplanes.
We record all cases where the intersection of the 4 hyperplanes with the Grassmannian $\operatorname{Gr}(2,5)$ is irreducible with singular locus of codimension greater than 1; there are 17 such intersections, of which 7 are smooth, corresponding to the fact that quintic del Pezzo surfaces over a finite field are indexed by conjugacy classes in $S_5$ (e.g., see \cite[Table 1]{Trepalin}).

For each of these 17 intersections, we first record all the quadrics defined on the span of the 4 linear forms, which reduces the enumeration of quadrics from a $\binom{10}{2}=55$ dimensional space to a $\binom{7}{2}=21$ dimensional space); we then record the cases where the intersection is smooth of genus 6.  The implementation of this method can be found in \verb|./Census/generic/| in our code base. 

\subsection{Postprocessing} \label{subsec:postprocessing}

For each stratum, the computation described above yields a finite set of curves of genus 6 over $\FF_2$ lying in that stratum and including at least one representative of each isomorphism class. It then remains to remove redundant representatives.

For this, we first hash the curves by their zeta function, or equivalently by the function $C \mapsto (\#C(\FF_{2^i})_{i=1}^6)$. Within each hash class, we use Magma to construct the function field of each curve, then use \verb+Isomorphisms+ to test whether any pair of curves is isomorphic. Once this is done, we compute the automorphism group of each curve that remains.

For the record, we mention some bugs in Magma that we had to work around.
\begin{itemize}
    \item For two function fields, the function \verb+Isomorphisms+ returns a list of all isomorphisms between the two fields, but in some cases with repeated entries. This causes \verb+AutomorphismGroup+ to yield errors in certain cases, for which we compute the group structure directly from the output of \verb+Isomorphisms+.
    \item For two function fields, the function \verb+IsIsomorphic+ sometimes returns False even when the two fields are isomorphic. We instead test whether \verb+Isomorphisms+ returns a nonempty list.
\end{itemize}
\section{Consistency checks}
\label{sec:completeness}

The proof of Theorem~\ref{T:main} implicitly depends on the correctness both of the relevant features of the underlying computational systems (SageMath and Magma) and of our implementation of the search strategy described above. It is thus highly desirable to perform some logically independent consistency checks of the resulting data. We describe several such checks here.

\subsection{Point counting on \texorpdfstring{$\M_6$}{M6}}

We first verify the numerical assertion \eqref{eq:full count over F2}.
By \cite[Corollary~1.6]{CanningLarson-Mg}, 
there exists a monic polynomial $P(T) \in \Z[T]$ of degree 15 such that $\#\M_6(\FF_q) = P(q)$ for every prime power $q$.
On account of the Lefschetz trace formula for Deligne--Mumford stacks \cite[Theorem~3.1.2]{Behrend},
it is a feasible but challenging computation to extract the exact polynomial by computing in the tautological ring of $\M_6$
as indicated (and implemented) in \cite{admcycles}.
\begin{thm} \label{thm:genus 6 point count}
For every prime power $q$, we have
\[
\#\M_6(\FF_q) = 
q^{15}+q^{14}+2q^{13}+q^{12}-q^{10}+q^3-1.
\]
In particular, $\#\M_6(\FF_2)= 68615$ as asserted in Theorem~\ref{T:main}.
\end{thm}
\begin{proof}
See  \cite{BergstromEtAl}.
\end{proof}

Given Theorem~\ref{thm:genus 6 point count}, one can give an alternate proof of Theorem~\ref{T:main} by independently checking the following two concrete assertions.
\begin{itemize}
    \item 
    For each tabulated curve $C$, the order of $\#\mathrm{Aut}(C)$ is no greater than the reported value.
    \item
    No two of the tabulated curves lying in the same stratum are isomorphic. (For an extra consistency check, we tested this in Magma also for pairs of curves not lying in the same stratum.)
\end{itemize}
Given these assertions, one may then directly verify from our data that $\#\M_6(\FF_q) \geq 68615$ with equality if and only if our census is complete. Combining with Theorem~\ref{thm:genus 6 point count} then yields Theorem~\ref{T:main}.

\subsection{Point counts with marked points}

As noted earlier, given Theorem~\ref{T:main} one can count the $\FF_2$-points of any moduli stack corresponding to genus 6 curves with some additional marked structure, as in Corollary~\ref{cor:count points on marked}. This count will always yield an integer thanks to the following fact.

\begin{lem}
    Let $\mathcal{X}$ be a Deligne--Mumford stack over a finite field $\FF_q$ admitting a coarse moduli space $X$. Then $\#\mathcal{X}(\FF_q) = \#X(\FF_q)$.
\end{lem}
\begin{proof}
    See \cite[Proposition 1.3(3)]{BergstromFaberPayne}.
\end{proof}

\subsection{Point counts in strata}

Point counts of some strata of $\M_6$ are also known, and can be used to check the corresponding sections of our table. See Table~\ref{table:point counts} for a summary of this discussion.

\begin{itemize}
    \item 
    For hyperelliptic curves, it is straightforward to compute that
\[
\#\mathcal{H}_6(\mathbb{F}_q) = q^{11};
\]
e.g., see  \cite{Bergstrom} for much stronger results.

\item 
For plane quintics, Gorinov \cite{Gorinov} showed that $\mathcal{Q}_6$ has trivial rational cohomology, yielding
\[
\#\mathcal{Q}_6(\mathbb{F}_q) = q^{12}.
\]
This has been rederived by elementary means by Wennink \cite{Wennink}.

\end{itemize}

We are not aware of any prior computation of  $\#\mathcal{T}_{6,n}(\FF_q)$. Comparing the values for $q=2$ with Zheng's results on the stable cohomology of $\mathcal{T}_{g,n}$ \cite{Zheng} suggests that
\[
\#\mathcal{T}_{6,0}(\FF_q) \approx q^{13} - q^{10}, \qquad
\#\mathcal{T}_{6,2}(\FF_q) \approx q^{12} + q^{11}.
\]
For $\#\mathcal{B}_6$, we have the following result for odd primes that does not appear to have been reported previously, but which does not yield a correct prediction for $q=2$ (see below).

\begin{prop} \label{P:bielliptic count}
For $6 \leq g \leq 11$, for every odd prime $q$,
\begin{equation} \label{eq:bielliptic count}
\#\mathcal{B}_g(\FF_q) = 
\frac{q^{2g}- q^{2g-4} - q^{2g-5} + (-1)^{g+1} q}{q^2+1}.
\end{equation}
\end{prop}
\begin{proof}
For $E$ an elliptic curve over $\mathbb{F}_q$, let $E^\circ$ denote the set of closed points of $E$ (of arbitrary degree)
and let $a(E) := q+1-\#E(\FF_q)$ be the trace of Frobenius of $E$.
For $n \geq 0$, let $d_n(E)$ denote the number of effective squarefree divisors of degree $n$ on $E$. We compute the generating series for $d_n(E)$ by writing
\begin{align*}
\sum_{n=0}^\infty d_n(E) T^n &= \prod_{x \in E^\circ} (1 + T^{\deg(x)}) = \prod_{x \in E^\circ} \frac{1-T^{2\deg(x)}}{1 - T^{\deg(x)}} \\
&= \frac{Z(X,T)}{Z(X,T^2)} = \frac{(1-T^2)(1-qT^2)(1-a(E)T+qT^2)}{(1-T)(1-qT)(1-a(E)T^2+qT^4)} \\
&= 1 + (q-a(E)+1)T \frac{1 - qT^3}{(1-qT)(1-a(E)T^2+qT^4)}.
\end{align*}

For any bielliptic curve $C$ of genus $g \geq 6$, by Castelnuovo--Severi the map from $C$ to its genus-1 quotient is unique up to composition by an automorphism of the target. In particular, the bielliptic involution $\iota$ of $C$ and central in $\mathrm{Aut}(C)$.

For a given elliptic curve $E$ (which as usual has a marked point $O$) and a given $g$, every bielliptic covering $C \to E$ of genus $g$ gives rise to a pair $(D, \mathcal{L})$ in which $D$ is an effective squarefree divisor (the branch locus) and $\mathcal{L}$ is a square root of the line bundle $\mathcal{O}(D)$. In particular, such a pair can only exist if the sum over $D$ yields an element of $2E(\FF_q)$; when this condition does hold, the square roots of $\mathcal{O}(D)$ form a torsor for the group $E(\FF_q)[2]$. Moreover, the bielliptic covering is determined by the pair up to a relative quadratic twist.

Putting this together, if we view $\M_{1,1}(\FF_q)$
as a measure space by weighting the isomorphism class of $E$ by $\frac{1}{\#\mathrm{Aut}(E)}$, then
\begin{align*}
\#\mathcal{B}_g(\mathbb{F}_q) &= 
\int_{\M_{1,1}(\FF_q)} \frac{d_{2g-2}(E)}{q-a(E)+1} \\
&= \int_{\M_{1,1}(\FF_q)} 
[T^{2g-3}] \frac{1 - qT^3}{(1-qT)(1-a(E)T^2+qT^4)},
\end{align*}
Since we are only extracting odd coefficients, we may rewrite this as
\begin{align*}
\#\mathcal{B}_g(\mathbb{F}_q) &= \frac{1}{2} \int_{\M_{1,1}(\FF_q)} 
[T^{2g-3}] \left( \frac{1 - qT^3}{1-qT} -\frac{1 + qT^3}{1+qT} \right) \frac{1}{1-a(E)T^2+qT^4} \\
&= q \int_{\M_{1,1}(\FF_q)} 
[T^{2g-3}] T \frac{1-T^2}{(1-q^2T^2)(1-a(E)T^2+qT^4)} \\
&= q [T^{g-2}]\int_{\M_{1,1}(\FF_q)}  \frac{1-T}{(1-q^2T)(1-a(E)T+qT^2)}.
\end{align*}
To evaluate the integral, we first recall that $\M_{1,1}(\FF_q)$ has total measure $q$. We next recall that elliptic curves over $\FF_q$ come in quadratic twist pairs whose Frobenius traces differ by a sign, so $\int_{\M_{1,1}(\FF_q)} a(E)^{2n+1} = 0$ for all $n \geq 0$.
We finally invoke a result of Birch \cite{Birch}: for $q$ an odd prime,
\begin{align*}
\int_{\M_{1,1}(\FF_q)}
a(E)^2 &= q^2 - 1 \\
\int_{\M_{1,1}(\FF_q)} 
a(E)^4 &= 2q^3 - 3q - 1 \\
\int_{\M_{1,1}(\FF_q)} 
a(E)^6 &= 5q^4 - 9q^2 - 5q - 1 \\
\int_{\M_{1,1}(\FF_q)} 
a(E)^8 &= 14q^5 - 28q^3 - 20q^2 - 7q - 1.
\end{align*}
This yields
\[
\int_{\M_{1,1}(\FF_q)} \frac{1}{1 - a(E)T + qT^2}\equiv 
q - T^2 - T^4 - T^6 - T^8 \pmod{T^{10}};
\]
hence for $6 \leq g \leq 11$,
\begin{align*}
\mathcal{B}_g(\FF_q) &= 
q [T^{g-2}] \frac{1-T}{1-q^2 T} \left( q - \frac{T^2}{1-T^2} \right) \\
&= q [T^{g-2}] \frac{q-(q+1)T^2}{(1+T)(1-q^2 T)} \\
&= \frac{q}{q^2+1} [T^{g-1}]  \left( \frac{q-(q+1)T^2}{1-q^2 T} - \frac{q-(q+1)T^2}{1+T} \right) \\
&= \frac{q}{q^2+1}(q q^{2g-2} - (q+1)q^{2g-6}
- q(-1)^{g-1} + (q+1) (-1)^{g-3})
\end{align*}
which simplifies to the stated expression.
\end{proof}

\begin{rmk}
One can extend Proposition~\ref{P:bielliptic count} to odd prime powers using Ihara's extension of Birch's formulas; see \cite[Theorem~2]{Kaplan-Petrow} for a compact statement.

In characteristic $2$, while the Birch--Ihara formula remains valid (e.g., because $\overline{\M}_{1,2g-2}$ is smooth over $\Z$), the description of double covers via Kummer theory does not. Moreover, the formula \eqref{eq:bielliptic count} does not hold for $q=2$: it predicts $\#\mathcal{B}_6(\FF_2)=742$, which is off by $2$ from the correct count.
\end{rmk}

\begin{rmk}
It is also shown in \cite{Birch} that  $\int_{M_{1,1}(\FF_q)} a(E)^{10}$ includes a nonzero contribution from the $\Delta$ modular form, and so $\#\mathcal{B}_{12}(\FF_q)$ is not a polynomial in $q$.
This loosely corresponds to the fact that the bielliptic locus of $\M_g$ has only tautological cycle classes for $g \leq 11$ \cite{CanningLarson-bielliptic} but not for $g=12$ \cite{vanZelm}.
\end{rmk}

\begin{table}
\caption{Point counts (unweighted and weighted) over $\FF_2$ of the various strata of $\M_6$. Of the formulas over $\FF_q$,  only those not ending in $\cdots$ are proven.}
\label{table:point counts}
\begin{tabular}{cccccccccccc}
Stratum & Unweighted & Weighted & \multicolumn{8}{c}{Weighted count over $\FF_q$ (empirical)}  \\
\hline
$\mathcal{H}_6$ & 4134 & 2048 & &&&&$q^{11}$\\
$\mathcal{B}_6$ & 1530 & 744 & &&&&& $q^{10}$ & $ - q^8$ & $+ \cdots$ \\
$\mathcal{Q}_6$ & 4204 & 4096 & &&& $q^{12}$\\
$\mathcal{T}_{6,0}$ & 7282 & 7166 & && $q^{13}$ &&& $ - q^{10}$ &&$ + \cdots$ \\
$\mathcal{T}_{6,2}$ & 6181 & 6148 &&&&$q^{12}$ &$ + q^{11}$ & && $+\cdots$\\
$\M_6^{\mathrm{BN}}$ & 48896 & 48413 & $q^{15}$ & $ + q^{14}$ & $+ q^{13}$ & $- q^{12}$ & $ -2 q^{11}$  &$- q^{10}$ & $+q^8$ & $+ \cdots$\\
\hline
$\M_6$ & 72227 & 68615 & $q^{15}$ & $+ q^{14}$ & $+ 2q^{13}$ & $ + q^{12}$ & & $ - q^{10} $ &&  $+ q^3-1$
\end{tabular}
\end{table}

\section*{Acknowledgments}
Thanks to Samir Canning for discussions about \cite{BergstromEtAl} and to David Roe for importing our data into LMFDB.

\bibliographystyle{amsalpha}
\bibliography{source.bib}
\end{document}